\documentclass[runningheads]{llncs}
\usepackage{graphicx}
\usepackage[all, 2cell]{xy}
\usepackage[all]{xy}
\usepackage{amsmath}
\usepackage{amssymb}
\newcommand{\down}[1]{\ensuremath{{\downarrow}\,#1}}
\newcommand{\up}[1]{\ensuremath{{\uparrow}\,#1}}
\begin{document}
\title{Reducts of relation algebras: \\ The aspects of axiomatisability and finite representability \thanks{The research is supported by the project MK-1184.2021.1.1.}}
%
%
\author{Daniel Rogozin \inst{1}\orcidID{0000-0002-6180-4323}}
\authorrunning{D. Rogozin}
%
\institute{Institute for Information Transmission Problems, Russian Academy of Sciences
\email{daniel.rogozin@serokell.io}}
\maketitle              
\begin{abstract}
In this paper, we show that the class of representable residuated semigroups has the finite representation property. That is, every finite representable residuated semigroup is representable over a finite base. This result gives a positive solution to Problem 19.17 from the monograph by Hirsch and Hodkinson \cite{hirsch2002relation}.

We also show that the class of representable join semilattice-ordered semigroups is pseudo-universal and it has a recursively
enumerable axiomatisation. For this purpose, we introduce representability games for join semilattice-ordered semigroups.

\end{abstract}
\section{Introduction}

Relation algebras are a kind of Boolean algebras with operators that provide algebraisation
of binary relations \cite{jonsson1951boolean}. The class of all relation algebras,
denoted as ${\bf RA}$, consists of algebras of the signature $\{ 0, 1, +, -, ;, {}^{\smile}, {\bf 1}' \}$, and all those algebras obey certain axioms.
The class of representable relation algebras, ${\bf RRA}$, consists of algebras isomorphic to set relation algebras. ${\bf RRA}$ is a subclass of ${\bf RA}$, but the converse inclusion does not hold.
That is, there exist non-representable relation algebras
\cite{lyndon1950representation}. Moreover, the class ${\bf RRA}$ is not a finitely axiomatisable variety \cite{monk1964representable} with neither Sahlqvist \cite{venema1997atom} nor canonical axiomatisation \cite{hodkinson2005canonical}. The problem of determining whether a given finite relation algebra
is representable is undecidable, see \cite{hirsch2001representability}.

For this reason, we are interested in reducts since one may extract more positive results in the aspects of decidability, representability, and finite axiomatisability. There are several results on reducts of relation algebras that have no finite axiomatisation. The examples of non-finitely axiomatisable classes are ordered monoids \cite{hirsch2005class}, distributive residuated lattices \cite{andreka1994lambek}, join semilattice-ordered semigroups \cite{andreka2011axiomatizability}, meet semilattice-ordered semigroups with converses \cite{hodkinson2000axiomatizability}, etc. On the other hand, such classes as representable residuated semigroups \cite{andreka1994lambek} and ordered domain algebras \cite{hirsch2013ordered} are finitely axiomatisable. There are also subsignatures for which the question of finite axiomatisability remains open, see, e. g., \cite{andreka2011axiomatizability}.

The other direction we discuss is related to finite representability. A finite algebra of relations has the finite representation property if it is isomorphic to some algebra of relations over a finite base. The investigation of this problem is of interest to study such aspects as decidability of membership of ${\bf R}(\tau)$ for finite structures. The finite representation property also implies recursivity of the class of all finite representable $\tau$-structures \cite{hirsch2004finite}, if the whole class is finitely axiomatisable.  Here, $\tau$ is a subsignature of operations and predicates definable in $\{ 0, 1, +, -, ;, {}^{\smile}, {\bf 1 } \}$. The examples of the class having the finite representation property are some classes of algebras \cite{hirsch2004finite} \cite{hirsch2013ordered} \cite{mclean2016finite}, the subsignature of which contains the domain and range operators. The other kind of algebras of binary relations having the finite representation property is semigroups with so-called demonic refinement has been recently studied by Hirsch and \v{S}emrl \cite{hirsch2021finite}, but the same authors have recently shown that semigroups with demonic joins fail to have the finite representation property \cite{9470509}.

There are subsignatures $\tau$ such that the class ${\bf R}(\tau)$ of representable reducts fails to have the finite representation property, for example, $\{;, \cdot\}$, see \cite[Theorem 4.1]{hirsch2021finite}. In general, (un)decidability of determining whether a finite relation algebra has a finite representation is an open question \cite[Problem 18.18]{hirsch2002relation}.

In this paper, we consider reducts of relation algebras the signature of which consists of composition, residuals, and the binary relation symbol that denotes partial ordering. That is, we study the class of representable residuated semigroups. We show that ${\bf R}(;, \setminus, /, \leq)$ has the finite representation property. As a result, Problem 19.17 of \cite{hirsch2002relation} has a positive solution. The solution is based on the Dedekind-MacNeille completions and relational representations of quantales. We embed a finite residuated semigroup into a finite quantale by mapping every element to its lower cone. After that, we apply the relational representation for quantales. As a result, the original finite residuated semigroup has a Zaretski-style representation \cite{zaretskii1959representation} and this satisfies the finite base requirement.

In the final section, we study the class of representable join semilattice-ordered semigroups, denoted as ${\bf R}(;,+)$. It is
already known that this class is not finitely axiomatisable \cite{andreka2011axiomatizability}.
We show that ${\bf R}(;,+)$ has a recursively enumerable axiomatisation. For that,
we define networks and representability games. This class is axiomatised with the axioms of join semilattice-ordered semigroups plus the countable set of universal formulas claiming that
$\exists$ has a winning strategy on every finite step. The question of finite representability for this class remains open, see \cite[Problem 2]{10.1007/978-3-030-88701-8_29}.

\section{Definitions}

\subsection{Relation algebras and their reducts}

Let us introduce some basic definitions related to relation algebras. See \cite[Section 3]{hirsch2002relation} to have more details.
\begin{definition} A relation algebra is an algebra $\mathcal{R} = \langle R, 0, 1, +, -, ;, {}^{\smile}, {\bf 1 }\rangle$ such that $\langle R, 0, 1, +, - \rangle$ is a Boolean algebra, $\langle R, ;, {\bf 1} \rangle$ is a monoid, and the following equations hold, for all $a, b, c \in R$:
    \begin{enumerate}
      \item $(a + b) ; c = (a ; c) + (b ; c)$,
      \item $a^{\smile \smile} = a$,
      \item $(a + b)^{\smile} = a^{\smile} + b^{\smile}$,
      \item $(a ; b)^{\smile} = b^{\smile} ; a^{\smile}$,
      \item $a^{\smile} ; (- (a ; b)) \leq - b$.
    \end{enumerate}
where $a \leq b$ is defined usually as $a + b = b$. ${\bf RA}$ is the class of all relation algebras.
\end{definition}

\begin{definition}
    A proper relation algebra (or, a set relation algebra) is an algebra $\mathcal{R} = \langle R, 0, 1, \cup, -, ;, {}^{\smile}, {\bf 1 }\rangle$ such that $R \subseteq \mathcal{P}(W)$, where $X$ is a base set, $W \subseteq X \times X$ is an equivalence relation, $0 = \emptyset$, $1 = W$, $\cup$ and $-$ are set-theoretic union and complement respectively, $;$ is relation composition, ${}^{\smile}$ is relation converse,
    ${\bf 1}'$ is the identity relation restricted to $W$, that is:
    \begin{enumerate}
    \item $a ; b = \{ (x, z) \in W \: | \: \exists y \: (x, y) \in a \: \& \: (y, z) \in b \}$
    \item $a^{\smile} = \{ (x, y) \in W \: | \: (y, x) \in a \}$
    \item ${\bf 1}' = \{ (x, y) \in W \: | \: x = y \}$
    \end{enumerate}
       ${\bf PRA}$ is the class of all proper relation algebras. ${\bf RRA}$ is the class of all representable relation algebras, that is, the closure of ${\bf PRA}$ under isomorphic copies.
\end{definition}

Let $\tau$ be a subset of operations and predicates definable in ${\bf RA}$. ${\bf R}(\tau)$ is the class of subalgebras of $\tau$-subreducts of algebras belonging to ${\bf RRA}$. We also assume that ${\bf R}(\tau)$ is closed under isomorphic copies. A $\tau$-structure is \emph{representable} if it is isomorphic to some algebra of relations of this signature. A representable finite $\tau$-structure has a \emph{finite representation over a finite base} if it is isomoprhic to some finite representable over a finite base. ${\bf R}(\tau)$ has the finite representation property if every $\mathcal{A} \in {\bf R}(\tau)$ has a finite representation over a finite base.

\subsection{Residuated semigroups}

A \emph{residuated semigroup} is a structure $\mathcal{A} = \langle A, ;, \leq, \setminus, / \rangle$ such that, for all $a, b, c \in \mathcal{A}$:

\begin{enumerate}
    \item $\leq$ is reflexive, antisymmetric, and transitive.
    \item $a ; (b ; c) = (a ; b) ; c$.
    \item $a \leq b \Rightarrow a ; c \leq b ; c$ and $a \leq b \Rightarrow c ; a \leq c ; b$.
    \item $b \leq a \setminus c \Leftrightarrow a ; b \leq c \Leftrightarrow a \leq c \: / \: b$.
\end{enumerate}

We can express residuals in every $\mathcal{R} \in {\bf RA}$ using Boolean negation, inversion, and composition as follows:

\begin{enumerate}
  \item $a \setminus b = -(a^{\smile} ; -b)$
  \item $a \: / \: b = - (- a ; b^{\smile})$
\end{enumerate}

These residuals have the following explicit definition in $\mathcal{R} \in {\bf PRA}$:
\begin{enumerate}
  \item $a \setminus b = \{ (x, y) \: | \: \forall z \: (z, x) \in a \Rightarrow (z, y) \in b \}$
  \item $a \: / \: b = \{ (x, y) \: | \: \forall z \: (y, z) \in b \Rightarrow (x, z) \in a \}$
\end{enumerate}

One can visualise residuals in ${\bf RRA}$ with the following triangles:

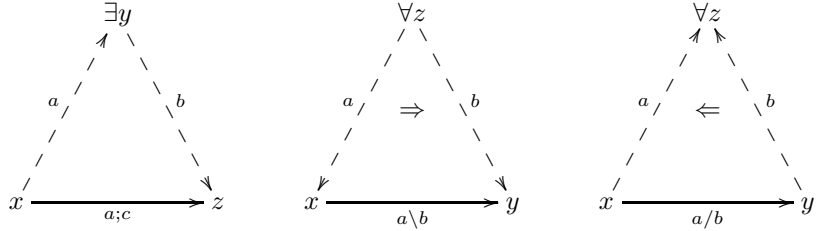
\xymatrix{
&& \exists y \ar@{-->}[ddr]^{b} &&& \forall z \ar@{-->}[ddl]_{a} \ar@{-->}[ddr]^{b} &&& \forall z \\
&&&&& \Rightarrow &&& \Leftarrow \\
& x \ar@{-->}[uur]^{a} \ar[rr]_{a;c} && z & x \ar[rr]_{a \setminus b} && y & x \ar@{-->}[uur]^{a} \ar[rr]_{a / b} && y \ar@{-->}[uul]_{b}
}

Thus, in particular, every relation algebra is a residuated lattice.

\subsection{Join semilattice-ordered semigroups}

A \emph{join semilattice-ordered semigroup} is an algebra $\mathcal{A} = \langle A, ;, + \rangle$ such that $\langle A, ; \rangle$ is a semigroup, $\langle A, + \rangle$ is a join-semilattice, and the following identities hold, for all $a, b, c \in A$:
\begin{enumerate}
\item $a ; (b + c) = a ; b + a ; c$,
\item $(a + b) ; c = a ; c + b ; c$.
\end{enumerate}
A join semilattice-ordered semigroup is also a poset and ordering is defined as $a \leq b$ iff $a + b = b$.

\begin{definition}\label{joinrep}
  A \emph{representation} $R$ of a join semilattice-ordered semigroup $\mathcal{A}$
  is a one-to-one map $R : \mathcal{A} \to 2^{D \times D}$ (where $D$ is a non-empty base set) such that
  \begin{enumerate}
    \item $(a + b)^R = a^R \cup b^R$,
    \item $(a ; b)^R = a^R ; b^R$.
  \end{enumerate}
\end{definition}

A join semilattice-ordered semigroup $\mathcal{A}$ is \emph{representable}, if there exists a representation $R : \mathcal{A} \to 2^{D \times D}$ for some non-empty base set $D$.

\subsection{Order-theoretic definitions}

Let us also remind the reader several order-theoretic notions, see \cite[Chapter 1]{davey2002introduction} for more details. Let $\langle P, \leq \rangle$ be a partial order. An upper cone generated by $x$ is the set $\up{x} = \{ a \in P \: | \: x \leq a \}$. Let $A \subseteq P$, then $\up{A} = \bigcup \limits_{x \in A} \up{x} = \{ a \in P \: | \: \exists x \in P \: x \leq a \}$. The set of all upper cones of a poset $\langle P, \leq \rangle$ is denoted as $\operatorname{Up}(P)$.
Given $a \in P$, the \emph{lower cone} generated by $a$ is a subset $\down{a} = \{ x \in P \: | \: x \leq a \}$. The lower cone generated by a subset is defined similarly.

A \emph{closure operator} on a poset $\langle P, \leq \rangle$ is a monotone map $j : P \to P$ such that for all $a \in P$ we have $a \leq j a = j j a$.

\subsection{Pseudo-elementary classes}

The following definitions are due to \cite[Section 9]{hirsch2002relation}.
Let $\mathcal{K}$ be a class of structures of a first-order signature $\mathcal{L}$. $\mathcal{K}$ is called a \emph{pseudo-elementary} class if there are:
\begin{enumerate}
\item a two-sorted language $\mathcal{L}^{s}$ with disjoint sorts ${\bf a}$ and ${\bf r}$ that contains all symbols of $\mathcal{L}$ as ${\bf a}$-sorted symbols,
\item an $\mathcal{L}^{s}$-theory $T$, the defining theory.
\end{enumerate}
such that $\mathcal{K} = \{ \mathcal{M}^{\bf a} \upharpoonright_{\mathcal{L}} \: | \: \mathcal{M} \models T \}$.
More generally, a pseudo-elementary class is a reduct of an elementary class, see \cite{eklof1977ultraproducts}.

A pseudo-elementary class is \emph{pseudo-universal} if
\begin{enumerate}
\item a function symbol in $\mathcal{L}^{s}$ that differs from copies of $\mathcal{L}$ ones takes values in sort ${\bf r}$,
\item Every sentence in $T$ is built from atomic formulas and negated-atomic formulas using $\vee$, $\land$, $\forall x^{\bf a}$, $\forall x^{\bf r}$, $\exists x^{\bf r}$.
\end{enumerate}

We are going to use the following fact to axiomatise the class of representable join semilattice-ordered semigroups, see \cite[Corollary 9.15, Theorem 9.28]{hirsch2002relation}:

\begin{theorem} \label{axiomatise}
$ $
\begin{enumerate}
\item If $\mathcal{K}$ is a pseudo-universal class, then $\mathcal{K}$ is elementary and universally axiomatisable.
\item Let $\mathcal{K} = \{ \mathcal{M}^{\bf a} \upharpoonright_{\mathcal{L}} \: | \: \mathcal{M} \models T \}$ be a pseudo-universal class of $\mathcal{L}$-structures, where $T$ is an $\mathcal{L}^{s}$-theory and $\mathcal{L}$, $\mathcal{L}^{s}$, $T$ are recursively enumerable.
Then there exists the set of $\mathcal{L}$-sentences $\{ \eta_n \}_{n < \omega}$ from $T$ such that $\mathcal{A} \in \mathcal{K}$ iff $\mathcal{A} \models \{ \eta_n \}_{n < \omega}$. That is, $\{ \eta_n \}_{n < \omega}$ axiomatises $\mathcal{K}$.
\end{enumerate}
\end{theorem}

\section{The finite representation property for residuated semigroups}

The problem we are interested in is the following \cite[Problem 19.17]{hirsch2002relation}:

\begin{center}
  Does ${\bf R}(;, \setminus, /, \leq)$ have the finite representation property?
\end{center}

The class ${\bf R}(;, \setminus, /, \leq)$ consists of the following structures, here is the explicit definition:

\begin{definition} \label{rrs}
  Let $A$ be a set of binary relations on some base set $W$ such that $R = \cup A$ is transitive and $W$ is a domain of $R$. A relational residuated semigroup is an algebra $\mathcal{A} = \langle A, ;, \setminus, /, \subseteq \rangle$ where, for all $a, b \in A$:
  \begin{enumerate}
    \item $a ; b = \{ (x, z) \: | \: \exists y \in W \: ((x, y) \in a \: \& \: (y, z) \in b) \}$,
    \item $a \setminus b = \{ (x, y) \: | \: \forall z \in W \: ((z, x) \in a \Rightarrow (z, y) \in b)\}$,
    \item $a \: / \: b = \{ (x, y) \: | \: \forall z \in W \: ((y, z) \in b \Rightarrow (x, z) \in a)\}$,
    \item $a \leq b$ iff $a \subseteq b$.
  \end{enumerate}
\end{definition}

A residuated semigroup is called \emph{representable} if it is isomorphic to some algebra that belongs to ${\bf R}(;, \setminus, /, \leq)$.

\begin{definition}
  Let $\tau = \{ ;, \setminus, /, \leq \}$, let $\mathcal{A}$ be a $\tau$-structure and $X$ a base set. An \emph{interpretation} $R$ over a base $X$ maps every $a \in \mathcal{A}$ to a binary relation $a^R \subseteq X \times X$. A \emph{representation} of $\mathcal{A}$ is an interpretation $R$ that interprets operations and $\leq$ as above.
\end{definition}

Andr\'{e}ka and Mikul\'{a}s proved the representation theorem for residuated semigroups (\cite{andreka1994lambek}) in the step-by-step fashion. See this paper to learn more about step-by-step representations in general \cite{hirsch1997step}.
The representation theorem for residuated semigroups obviously implies that the class ${\bf R}(;, \setminus, /, \leq)$ is finitely axiomatisable. As it is well known, the logic of such structures is the Lambek calculus \cite{lambek1958mathematics}, so we also have the relational completeness of the Lambek calculus. With our result, we also have a version of the finite model property for the Lambek calculus since one can refute any unprovable sequent in some finite relational residuated semigroup over a finite base. This is a corollary of our result and the fact that the Lambek calculus is complete w.r.t finite residuated semigroups, see \cite[Chapter 7, Section 7.4]{galatos2007residuated} to have an even stronger result.

It is sufficient to show that any finite residuated semigroup has a representation over a finite base in order to show that ${\bf R}(;, \setminus, /, \leq)$ has the finite representation property. For that, we will use the representation of residuated
semigroups as subalgebras of quantales and the relational representation of quantales.

A quantale is a complete lattice-ordered semigroup. That is, a binary operation respects suprema in both arguments. Quantales have been introduced by Mulvey to provide a noncommutative generalisation of locales, see \cite{mulvey1986suppl}.

\begin{definition}
  A \emph{quantale} is a structure $\mathcal{Q} = \langle Q, ;, \Sigma \rangle$ such that $\mathcal{Q} = \langle Q, \Sigma \rangle$ is a complete lattice, where $\Sigma$ denotes an infinite join, $\langle Q, ; \rangle$ is a semigroup, and the following conditions hold for all $a \in Q$ and $A \subseteq Q$:
  \begin{enumerate}
    \item $a \: ; \: \Sigma A = \Sigma \{ a ; q \: | \: q \in A \}$,
    \item $\Sigma A \: ; \: a = \Sigma \{ q ; a \: | \: q \in A \}$.
  \end{enumerate}
\end{definition}

\begin{definition}\label{gen}
Given a quantale $\mathcal{Q} = \langle Q, ;, \Sigma \rangle$, a set of \emph{generators} is a subset $G \subseteq \mathcal{Q}$, if
\begin{enumerate}
  \item For all $q \in Q$ one has $q \leq \Sigma \{ g \in G \: | \: g \leq q \}$,
  \item For all $g \in G$ and $q_1, q_2 \in \mathcal{Q}$, $g \leq q_1 ; q_2$ implies $g \leq q_1 ; r$ for some $r \in G$ with $r \leq q_2$.
\end{enumerate}
\end{definition}

The existence of a set of generators for an arbitrary quantale has been shown here \cite[Lemma 3.12]{brown1993representation}.

Note that any quantale is a residuated semigroup as well. Given a quantale $\mathcal{Q} = \langle Q, ;, \Sigma \rangle$, One may express residuals with supremum and product as follows for all $a, b \in Q$:
\begin{enumerate}
  \item $a \setminus b = \Sigma \{ c \in Q \: | \: a ; c \leq b \}$,
  \item $a \: / \: b = \Sigma \{ c \in Q\: | \: b ; c \leq a \}$.
\end{enumerate}
It is readily checked that residuals are unique.

A quantic nucleus is a closure operator on a quantale. Such an operator is a noncommutative generalisation of a nucleus operator from locale theory. The following definition and the proposition below are due to \cite[Definition 3.1.1, Theorem 3.1.1]{rosenthal1990quantales} respectively.
\begin{definition}
  A \emph{quantic nucleus} on a quantale $\langle A, ;, \Sigma \rangle$ is a mapping $j : A \to A$ such that $j$ a closure operator satisfying $j a ; j b \leq j (a ; b)$.
\end{definition}

\begin{proposition} \label{subsemi}
  Let $\mathcal{A} = \langle A, ;, \Sigma \rangle$ be a quantale and $j$ a quantic nucleus, the set
  $\mathcal{A}_j = \{ a \in A \: | \: j a = a \}$ forms a quantale, where $a ;_j b = j(a ; b)$ and $\Sigma_j A = j (\Sigma A)$ for all $a, b \in {A}_j$ and $A \subseteq \mathcal{A}_j$.
\end{proposition}

One can embed any residuated semigroup into some quantale with using Dedekind-MacNeille completion (see, for example, \cite{theunissen2007macneille}) as follows. According to Goldblatt \cite{goldblatt2006kripke}, residuated semigroups have the following representation based on quantic nuclei and the Galois connection. We need the construction for the solution, so we discuss it briefly. See Goldblatt's paper to have a complete argument in more detail \cite{goldblatt2006kripke}.

Let $\mathcal{A} = \langle A, \leq, ;, \setminus, / \rangle$ be a residuated semigroup. Then $\langle \mathcal{P}(A), ;, \bigcup \rangle$ is a quantale with pairwise product of subsets.

Let $X \subseteq A$. We put $lX$ and $uX$ as the sets of lower and upper bounds of $X$ in $A$. We also put $m X = lu X$.
Note that the lower cone of an arbitrary $x$ is $m$-closed, that is, $m (\down{x}) = \down{x}$.

$m : \mathcal{P}(A) \to \mathcal{P}(A)$ is a closure operator and the set

\begin{center}
$(\mathcal{P}(A))_m = \{ X \in \mathcal{P}(S) \: | \: m X = X\}$
\end{center}
forms a complete lattice with $\Sigma_{m} \mathcal{X} = m ( \bigcup \mathcal{X})$ and $\Pi_{m} = \bigcap \mathcal{X}$, see \cite[Theorem 7.3]{davey2002introduction}. The key observation is that $m$ is a quantic nucleus on $\mathcal{P}(A)$, that is, $m A ; m B \subseteq m (A ; B)$. We refer here to the aforementioned paper by Goldblatt. Thus, according to Proposition~\ref{subsemi}, $\langle (\mathcal{P}(A))_m, \subseteq, ;_m \rangle$ is a quantale itself since $m$ is a quantic nucleus.

We define a map $f_m : \mathcal{A} \to (\mathcal{P}(A))_m$ such that $f_m : a \mapsto \down{a}$. This map is well-defined since any lower cone generated by a point is $m$-closed. Moreover, $f_m$ preserves products, residuals, and existing suprema. In particular, $f_m$ is a residuated semigroup embedding. As a result, we have the following representation theorem \cite[Corollary 2]{goldblatt2006kripke}.

\begin{theorem} \label{orsRep}
  Every residuated semigroup is isomorphic to the subalgebra of some quantale.
\end{theorem}

In turn, quantales are representable with quantales of binary relations. The notion of a relational quantale has been introduced by Brown and Gurr to represent quantales as quantales of relations \cite{brown1993representation}.
\begin{definition}
  Let $A$ be a non-empty set. A \emph{relational quantale} on $A$ is an algebra $\langle R, \subseteq, ; \rangle$, where
  \begin{enumerate}
    \item $R \subseteq \mathcal{P}(A \times A)$,
    \item $\langle R, \subseteq \rangle$ is a complete join-semilattice,
    \item $;$ is a relational composition that respects all suprema in both coordinates.
  \end{enumerate}
\end{definition}

The uniqueness of residuals in any quantale implies the following fact.
\begin{proposition}\label{ok}
  Let $\mathcal{A}$ be a relational quantale over a base set $X$, then for all $a, b \in \mathcal{A}$
  \begin{enumerate}
    \item $a \setminus b = \{ (x, y) \in X^2 \: | \: \forall z \in X ( (z, x) \in a \Rightarrow (z, y) \in b) \}$,
    \item $a \: / \: b = \{ (x, y) \in X^2 \: | \: \forall z \in X ( (y, z) \in b \Rightarrow (x, z) \in b )\}$.
  \end{enumerate}
\end{proposition}

Now let us discuss the representation theorem for quantales. Let $\mathcal{Q}$ be a quantale, $Q$ its carrier, and $\langle G \rangle$ a set of its generators. Given $a \in \mathcal{Q}$, define the binary relation $\hat{a} \subseteq Q \times Q$ as:

\begin{center}
  $\hat{a} = \{ (g,p) \: | \: g \in \langle G \rangle, p \in Q \:\: g \leq a ; p \}$
\end{center}
Denote $\widehat{Q}$ as $\{ \hat{a} \: | \: a \in \mathcal{Q} \}$.

The mapping $a \mapsto \hat{a}$ satisfies the following conditions:

\begin{enumerate}
\item $a \leq b$ iff $\hat{a} \subseteq \hat{b}$,

\item $\widehat{\Sigma A} = \Sigma \widehat{A}$, $\hat{a} ; \hat{b} = \widehat{a ; b}$, and $\langle \widehat{\mathcal{Q}}, \subseteq, \Sigma \rangle$ is a complete lattice,

\item $\langle \widehat{\mathcal{Q}}, \subseteq, ; \rangle$ is a relational quantale,

\item $\mathcal{Q}$ is isomorphic to $\langle \widehat{\mathcal{Q}}, \subseteq, ; \rangle$ and $a \mapsto \hat{a}$ is a quantale isomorphism.
\end{enumerate}

We summarise the construction above with the following theorem proved by Brown and Gurr, see \cite[Theorem 3.11]{brown1993representation}.

\begin{theorem} \label{quantaleRep}
  Every quantale $\mathcal{Q} = \langle Q, ;, \Sigma \rangle$ is isomorphic to a relational quantale on $Q$ as a base set.
\end{theorem}

Let $\mathcal{A}$ be a residuated semigroup and $\mathcal{Q}_{\mathcal{A}}$ a quantale of Galois closed subsets of $\mathcal{A}$. $\widehat{\mathcal{Q}_{\mathcal{A}}}$ is the corresponding relational quantale. Let us define an interpretation $R : \mathcal{A} \to \widehat{\mathcal{Q}_{\mathcal{A}}}$ such that:

\begin{center}
  $R : a \mapsto a^{R} = \widehat{\down{a}}$
\end{center}

According to the lemma below, such an interpretation is a representation. As we have already said above, the function $a \mapsto \down{a}$ is order-preserving and it commutes with products and residuals.

\begin{lemma} \label{interp}
  Let $\mathcal{A}$ be a residuated semigroup, then the interpretation $R : \mathcal{A} \to \widehat{\mathcal{Q}_{\mathcal{A}}}$ such that $R : a \mapsto a^{R} = \widehat{\down{a}}$ is a representation.
\end{lemma}

\begin{proof}
  By Theorem~\ref{orsRep}, $\mathcal{A}$ emdeds to $\mathcal{Q}_{\mathcal{A}}$, but by Theorem~\ref{quantaleRep}, $\mathcal{Q}_{\mathcal{A}}$ is isomorphic to $\widehat{\mathcal{Q}_{\mathcal{A}}}$.
  The fact that $R$ is an injective homomorphism follows from the construction of the embedding of a residuated semigroup to the quantale of its Galois-stable subsets, the isomorphism of $\mathcal{Q}_{\mathcal{A}}$ with $\widehat{\mathcal{Q}_{\mathcal{A}}}$, and Proposition~\ref{ok}.
\end{proof}

The lemma above imply the following statement.
\begin{theorem} \label{solution}
  Every residuated semigroup is isomorphic to the subalgebra of some relational quantale. Moreover, ${\bf R}(;, \setminus, /, \leq)$ has the finite representation property.
\end{theorem}

\begin{proof} Let $\mathcal{A}$ be a finite residuated semigroup. The representation of $\mathcal{A}$ as a subalgebra of the relational quantale of $\widehat{\mathcal{Q}_{\mathcal{A}}}$ belongs to ${\bf R}(;, \setminus, /, \leq)$ by Lemma~\ref{interp}. This representation has the following form:

\begin{center}
  $\widehat{\mathcal{A}} = \langle \{ \widehat{\down{a}} \}_{a \in \mathcal{A}}, ;, \setminus, /, \subseteq \rangle$.
\end{center}

  Moreover, such a representation with the corresponding relational quantale has the finite base, if the original algebra is finite. The base set of the quantale $\widehat{\mathcal{Q}_{\mathcal{A}}}$ is the set of Galois stable subsets of $\mathcal{A}$, which is finite.
\end{proof}

\section{Join semilattice-ordered semigroups: the explicit axiomatisation}

We  note  that  a  similar  construction  does  not  work  for  finite  representable upper semilattice-ordered semigroups. From the one hand, the notions of a finite upper semilattice-ordered semigroup and finite quantale are quite close to each other. From the other hand, the relational representation of quantales does not have to represent joins as set-theoretic unions generally. Moreover, there is a countable sequence of non-representable upper semilattice-ordered semigroups with a non-representable ultraproduct, see \cite[Theorem 3.1]{andreka2011axiomatizability}. Thus, ${\bf R}(;,+)$ is not finitely axiomatisable. Although, as we will see below, this class has a universal recursively enumerable axiomatisation. For that, we characterise representability using representability games on networks. The construction is somewhat similar to the proof of \cite[Proposition 5]{hirsch2005class}.

\begin{definition} Let $\mathcal{A}$ be a join-semilattice ordered semigroup. A \emph{prenetwork} over $\mathcal{A}$ is a tuple $(V, E, l)$, where $V$ is a set of vertices, $E$ is a set of edges such that $\langle V, E \rangle$ is a directed graph, and $l$ is a labelling function $l : E \to \operatorname{Up}(\mathcal{A})$.

  A prenetwork over $\mathcal{A} = (V, E, l)$ is a \emph{network} if the following hold:
  \begin{enumerate}
    \item {\bf (Saturation condition)} For all $u, v \in V$ and for all $x,y,z \in \mathcal{A}$, $z \in l(u, v)$ and $z \leq x \: ; \: y$ implies $x \in l(u, w)$ and $y \in l(w, v)$ for some $w \in V$.
    \item {\bf (Coherence condition)} For all $u, v, w \in V$, one has $l(u, v) ; l(v, w) \subseteq l(u, w)$.
    \item {\bf (Join-primeness)} For all $u, v \in V$, $l(u,v)$ is join-prime. That is, for all $a, b \in \mathcal{A}$ if $a + b \in l(u,v)$, then either $a \in l(u,v)$ or $b \in l(u,v)$.
  \end{enumerate}
\end{definition}

If $\mathcal{N}$ is a prenetwork, then we will denote its sets of nodes as $\operatorname{Nodes}(\mathcal{N})$ occasionally.

Let $I$ be a non-empty index set and let $\{ \mathcal{N}_i \}_{i \in I}$ be an indexed set of prenetworks (where each $\mathcal{N}_i = (V_i, E_i, l_i)$), then $\mathcal{N} = \bigcup \limits_{i \in I} \mathcal{N}_i$ defined as $(V, E, l)$, where

\begin{enumerate}
  \item $V = \bigcup \limits_{i \in I} V_i$ and $E = \bigcup \limits_{i \in I} E_i$.
  \item $l(x, y) = \bigcup \limits_{ i \in I } l_i(x, y)$ for all $x, y \in V$.
\end{enumerate}

\begin{definition}
  Let $n \leq \omega$ and $\mathcal{A}$ a join semilattice-ordered semigroup. A play of the game $\mathcal{G}_n(\mathcal{A})$ has $n$ rounds and consists of $n$ prenetworks. As usual, we have two players, $\forall$ (Abelard, he/his) and $\exists$ (H\'{e}lo\"{i}se, she/her).

  \begin{enumerate}
    \item Round $0$: $\forall$ picks $a, b \in \mathcal{A}$ such that $a \not\leq b$. $\exists$ responds with a prenetwork $\mathcal{N}_0 = (V_0 = \{ x_0, x_1 \}, E_0 = \{ (x_0, x_1)\}, l_0)$ such that $l_0(x_0, x_1) = \up a$.
    \item Round $n + 1$. Suppose, the prenetwork $\mathcal{N}_n = (V_n, E_n, l_n)$ has been played.

    $\forall$ has the following three options:
    \begin{enumerate}
      \item {\bf (Composition move)}: $\forall$ picks $x, y, z \in V_n$ with $b \in l_n(x, y)$ and $c \in l_n(y, z)$. We denote such a move as $N(x,y,z,b,c)$. Then $\exists$ responds with $\mathcal{N}_{n + 1} = (V_{n + 1}, E_{n + 1}, l_{n + 1})$ such that $\mathcal{N}_{n + 1}$
      is the same as $\mathcal{N}_n$, but $l_{n + 1}(x, z) = \up{(l_{n}(x, z)} \cup \{ b \: ; \: c \})$.
      \item {\bf (Witness move)}:

      $\forall$ picks an edge $(x, y) \in E_n$ and $d, e \in \mathcal{A}$ such that $c \leq d ; e$ for $c \in l_n(x, y)$. $\exists$ has to find a witness. She has to find a $z$ which is either a fresh node or an old one. If $z$ is fresh, then she defines the prenetwork $T$, the edges of which are $x, y, z$ with labelling:
      \begin{enumerate}
        \item $l_T(x, z) = \up{d}$
        \item $l_T(z, y) = \up{e}$
      \end{enumerate}
      If $z$ is already an element of $\mathcal{A}$, then her response is similar.
      For her response, $\exists$ plays $\mathcal{N}_{n + 1} = \mathcal{N}_n \cup T$.
      \item {\bf (Join move)}:

      $\forall$ picks an edge $(x, y) \in E_n$ and $c + d$ for $c, d \in \mathcal{A}$. $\exists$ has the following two alternatives for her response:
      \begin{enumerate}
        \item $\exists$ chooses $c$ and responds with the prenetwork $\mathcal{N}_{n + 1} = \langle V_{n + 1}, E_{n + 1}, l_{n + 1} \rangle$, where $l_{n+1}(x, y) = \up{(l_n(x,y) \cup \{ c \})}$.
        \item $\exists$ chooses $b$. The response is similar but $l_{n + 1}(x, y) = \up{(l_n(x,y) \cup \{ d \})}$.
      \end{enumerate}
    \end{enumerate}
  \end{enumerate}

  $\forall$ wins the play if $b \notin l_{\mathcal{N}_i}(x,y)$ for some $i < n$. Otherwise, $\exists$ wins the play.

  Let $a \in \mathcal{A}$ and $\mathcal{N}$ a network, define a game $\mathcal{G}(\mathcal{N}, \mathcal{A}, a)$ such that $\forall$ picks $a$ in the initial round and $\mathcal{N}_0 = \mathcal{N}$. The rules of the game are the same as previously.
\end{definition}

\begin{lemma} \label{repr} Let $\mathcal{A} = \langle A, ;, + \rangle$ be a join semilattice-ordered semigroup,

  \begin{enumerate}
    \item If $\mathcal{A}$ is representable then $\exists$ has a winning strategy in $\mathcal{G}_{\omega}(\mathcal{A})$.
    \item If $|\mathcal{A}| \leq \omega$ and $\exists$ has a winning strategy in $\mathcal{G}_{\omega}(\mathcal{A})$ then $\mathcal{A}$ is representable.
  \end{enumerate}
\end{lemma}

\begin{proof}
$ $

  \begin{enumerate}
    \item Let $h : \mathcal{A} \to 2^{D \times D}$ be a representation of some base set $D \neq \emptyset$. $\exists$ maintains a map ${}^{'} : \operatorname{Nodes}(\mathcal{N}) \to D$, where $\mathcal{N}$ is a network being played, such that $a \in l_{\mathcal{N}}(x,y)$ implies $(x',y') \in h(a)$.

    \item
    Given $a \in \mathcal{A}$, we consider a play of the game where $\forall$ picks $a$ and $b$ with $a \not\leq b$ in the initial round and plays $(\mathcal{N}, x, y, z, c, d)$ in the further rounds for all $x, y, z \in \operatorname{Nodes}(\mathcal{N})$ and $c,d \in \mathcal{A}$. Here, $c \in l_N(x, y)$ and $d \in l_N(y, z)$.

    $\forall$ also plays all rounds $(\mathcal{N},x,y,c,d)$ for all $x, y \in \operatorname{Nodes}(\mathcal{N})$ and $c,d \in \mathcal{A}$ such that there is $e \in \mathcal{A}$ such that $e \leq c ; d$ and $e \in l_N(x,y)$.

    $\forall$ picks also $c + d$ and vertices $x,y \in \operatorname{Nodes}(\mathcal{N})$ for $c, d \in \mathcal{A}$.

    Note that $\mathcal{A}$ is at most countable, so we can schedule all these moves. We have the following play of a game where H\'{e}lo\"{i}se uses a winning strategy:

    \begin{center}
      $\mathcal{N}_0 \subseteq \mathcal{N}_1 \subseteq \mathcal{N}_2 \dots$
    \end{center}

    Let us put $\mathcal{N}^{*}(a, b) = \bigcup \limits_{i < \omega} \mathcal{N}_i$. $\mathcal{N}^{*}(a,b)$ is clearly a network. Let us put the following network assuming that $\mathcal{N}^{*}(a_1, a_2)$ and $\mathcal{N}^{*}(b_1, b_2)$ are disjoint for $a_1 \neq a_2$ and $b_1 \neq b_2$:

    \begin{center}
      $\mathcal{N} = \bigcup \limits_{a,b \in \mathcal{A}, a \not\leq b} \mathcal{N}^{*}(a,b)$
    \end{center}

    Note that $\mathcal{N} = \langle V, E, l \rangle$ is a well-defined network since it is the disjoint union of networks.

    Define $\operatorname{rep} : \mathcal{A} \to E$ as:
    \begin{center}
      $\operatorname{rep}(a) = \{ (x, y) \: | \: \exists b \leq a \:\: b \in l_{\mathcal{N}}(x, y)\}$
    \end{center}

    Let us check that $\operatorname{rep}$ is a representation. Let us show that $\operatorname{rep}(a + b) = \operatorname{rep}(a) \cup \operatorname{rep}(b)$ Suppose $(x, y) \in \operatorname{rep}(a + b)$. That is, there exists $c \leq a + b$ with $c \in l_{\mathcal{N}}(x, y)$, so does $a + b$ since $l_{\mathcal{N}}$ is an upper cone. $a + b \in l_{\mathcal{N}}(x, y)$, that is,

\begin{center}
$a + b \in \bigcup \limits_{ \substack{c_1, c_2 \in \mathcal{A} \\ c_1 \not\leq c_2}} l_{\mathcal{N}^{*}(c_1, c_2)}(x,y)$
\end{center}

That is, there is $c \in \mathcal{A}$ with such that $a + b \in l_{\mathcal{N}^{*}(c_1, c_2)}(x,y)$, but $l_{\mathcal{N}^{*}(c_1, c_2)}(x,y)$ is join-prime, so we have either $a \in l_{\mathcal{N}^{*}(c_1, c_2)}(x,y)$ or $b \in l_{\mathcal{N}^{*}(c_1, c_2)}(x,y)$. Thus, $\operatorname{rep}(a + b) \subseteq \operatorname{rep}(a) \cup \operatorname{rep}(b)$.

Suppose for the converse, $(x, y) \in \operatorname{rep}(a)$. We need $(x, y) \in \operatorname{rep}(a + b)$.
In other words, we have some $c \in \mathcal{A}$ with $c \leq a$ and $c \in l_{\mathcal{N}}(x, y)$.
We have $c \leq a \leq a + b$, so $(x, y) \in \operatorname{rep}(a + b)$.

Let us show that $\operatorname{rep}(a ; b) = \operatorname{rep}(a) ; \operatorname{rep}(b)$.

Suppose $(x, y) \in \operatorname{rep}(a ; b)$. We need some $z$ with $(x, z) \in \operatorname{rep}(a)$ and $(z, y) \in \operatorname{rep}(b)$. There is $c \leq a ; b$ with $c \in l_{\mathcal{N}}(x,y)$. That is, there are $a_1, a_0 \in \mathcal{A}$ and $\mathcal{N}_i$ such that $c \in l_{\mathcal{N}_i}(x, y)$ where $\forall$ plays $(a_1, a_0)$ for the initial round. By the condition, $\forall$ makes the witness moves and $\exists$ responds with a witness. Her response is a node $z$ such that $l_{\mathcal{N}_{i + 1}}(x, z) =  \up{(l_{\mathcal{N}_i}(x, z) \cup \{ a \})}$ and $l_{\mathcal{N}_{i + 1}}(z, y) = \up{(l_{\mathcal{N}_i}(z, y) \cup \{ b\})}$.
The inclusion $\operatorname{rep}(a ; b) \subseteq \operatorname{rep}(a) ; \operatorname{rep}(b)$ holds since all witness moves have been played.

Suppose $(x, y) \in \operatorname{rep}(a) ; \operatorname{rep}(b)$. We need $(x, y) \in \operatorname{rep}(a;b)$. There exists $z \in \operatorname{Nodes}(\mathcal{N})$ with $(x, z) \in \operatorname{rep}(a)$ and $(z, y) \in \operatorname{rep}(b)$. So, there are $c, d$ such that $c \leq a$ with $c \in l_{\mathcal{N}}(x, z)$ and $d \leq b$ with $d \in l_{\mathcal{N}}(z, y)$.
We also know that $l_{\mathcal{N}}(x, z) ; l_{\mathcal{N}}(z, y) \subseteq l_{\mathcal{N}}(x, y)$ because all composition moves have been played.
So $c ; d \in l_{\mathcal{N}}(x, y)$. That makes $(x, y) \in \operatorname{rep}(a;b)$ since $c ; d \leq a ; b$.

For injectivity, suppose $a \leq b$ and $(x, y) \in \operatorname{rep}(a)$, that is, there is $c \leq a$ such that $c \in l_{\mathcal{N}}(x, y)$, but $c \leq a \leq b$, so $(x, y) \in \operatorname{rep}(b)$.

Suppose $a \not\leq b$, then there are $x, y \in \operatorname{Nodes}(\mathcal{N}(a, b))$ such that $a \in l_{\mathcal{N}}(x, y)$ and $b \notin l_{\mathcal{N}}(x, y)$. These elements are $x_0, x_1$ that $\exists$ picks as her response in the zero round. $\exists$ has a winning strategy, so $b \notin l(x_0, x_1)$, but $(x,y) \in \operatorname{rep}(a)$, but $(x,y) \notin \operatorname{rep}(b)$.
  \end{enumerate}
\end{proof}

The following proposition is a version of \cite[Proposition 7.24]{hirsch2002relation} and the right-to-left part is proved using K\"{o}nig's lemma \cite[Exercise 5.6.5]{hodges1993model}.

\begin{proposition} \label{fin}
  Let $\mathcal{A}$ be a join semilattice-ordered semigroup and $\mathcal{N}$ a network, iff $\exists$ has a winning strategy in $\mathcal{G}_n(\mathcal{A}, \mathcal{N})$ for all $n < \omega$ iff she has a winning strategy in $\mathcal{G}_{\omega}(\mathcal{A}, \mathcal{N})$.
\end{proposition}

Our purpose is to axiomatise axiomatisation of ${\bf R}(;, +)$ with a recursively enumerable set of universal formulas. See \cite[Chapter 9]{hirsch2002relation} for the discussion in detail to have a more general methodology.

\begin{definition}
  Let $\operatorname{Var} = \{ v_0, v_1, \dots \}$ be a set of variables. The set of terms is generated by the following grammar:
  \begin{center}
    $t_1, t_2 ::= v \: | \: (t_1 + t_2) \: | \: (t_1 ; t_2)$
  \end{center}
\end{definition}

\begin{definition}
  A \emph{term network} is a finite network $\langle V, E, l \rangle$, where $\langle V, E \rangle$ is a directed graph and $l : E \to 2^{Term}$ is a labelling function such that every $l(x,y)$ is finite for all $(x, y) \in E$.
\end{definition}

Let $\mathcal{A}$ be a join semilattice-ordered semigroup and $\vartheta : \operatorname{Var} \to \mathcal{A}$ a valuation. The value of complex terms is defined inductively for $a, b \in T$:

\begin{enumerate}
  \item $(a ; b)^{\vartheta} = a^{\vartheta} ; b^{\vartheta}$
  \item $(a + b)^{\vartheta} = a^{\vartheta} + b^{\vartheta}$
\end{enumerate}

Let $\mathcal{N} = \langle V, E, l \rangle$ be a term network, $\mathcal{A}$ be a join-semilattice ordered semigroup and $\vartheta : \operatorname{Var} \to {\mathcal{A}}$ a valuation. Let us define the prenetwork $\mathcal{N}^{\vartheta}$ with the same edges and vertices with labelling $l^{\vartheta}(x, y) = \up \vartheta[l_{\mathcal{N}}(x, y)]$. We define the following three extensions of $\mathcal{N}$ reflecting the composition, witness, and join moves respectively:
\begin{enumerate}
  \item Let $x, y \in \operatorname{Nodes}(\mathcal{N})$ and let $t$ be a term. $\mathcal{N}_c$ is the extension of $\mathcal{N}$,
  where $\operatorname{Nodes}(\mathcal{N}_c) = \operatorname{Nodes}(\mathcal{N})$ and $l_{\mathcal{N}_c}(x, y) = l_{\mathcal{N}}(x, y) \cup \{ t\}$ and $l_{\mathcal{N}_c}(u, v) = l_{\mathcal{N}}(u, v)$ for all $u \neq x$ and $v \neq y$. We denote this network as $\mathcal{N}_c(\mathcal{N}, x, y, t)$.

  \item
  Let $x, y \in \operatorname{Nodes}(\mathcal{N})$, let $z$ be a node (regardless of whether $z$ is fresh or not), and $t_1$, $t_2$ any terms. Let us define a network $T$ such that $\operatorname{Nodes}(T) = \{ x, y, z\}$. We define labelling as $l_{T}(x, y) = \{ t_1 \}$ and $l_{T}(y, z) = \{ t_2 \}$. So we put $\mathcal{N}_w = \mathcal{N} \cup T$.
  We denote this network as $\mathcal{N}_w(\mathcal{N}, x,y,z, t_1, t_2)$.

  \item Let $x, y \in \operatorname{Nodes}(\mathcal{N})$ and let $t_1, t_2$ be terms. We define $T_i = \langle \{ x, y\}, \{ (x, y) \}, l_{T_i} \rangle$, where $l_{T_i}(x, y) = l_{\mathcal{N}}(x, y) \cup \{ t_i \}$ for $i = 1,2$. So $\mathcal{N}_{j_1} = \mathcal{N} \cup T_1$ and $\mathcal{N}_{j_2} = \mathcal{N} \cup T_2$.
\end{enumerate}

\begin{lemma}\label{ax}
  For all $n < \omega$ there exists a first-order sentence $\rho_n$ such that $\exists$ has a winning strategy in $\mathcal{G}_n(\mathcal{A})$ iff $\mathcal{A} \models \rho_n$.
\end{lemma}

\begin{proof}
  As usual, for each $n < \omega$ we construct a formula $\sigma_n$ claiming that $\exists$ has a winning strategy in the game of lenght $n$. To be more precise, our purpose is to have
  \begin{center}
    $\exists$ has a winning strategy in $\mathcal{G}_n(\mathcal{N}^{\vartheta}, \mathcal{A}, \vartheta(v))$ if and only if $\mathcal{A} \models \sigma_{n}(\mathcal{N}, v)$
  \end{center}
  where $\mathcal{A}$ is a join semilattice-ordered semigroup, $\vartheta : \operatorname{Var} \to \mathcal{A}$ is a variable assignment, and $\mathcal{N}$ is a term network.

  We define the following sequence of formulas $\{ \sigma_n \}_{n < \omega}$ inductively:
  \begin{enumerate}
    \item $\sigma_0(\mathcal{N}, v) = \bigwedge \limits_{a \in l_{\mathcal{N}}(x, y)} \neg (a \leq v)$

$\sigma_0(\mathcal{N}, v)$ merely claims that $\exists$ has a winning strategy in the zero length game.
    \item Suppose $\sigma_{n}(\mathcal{N}, v)$ are already constructed for some $n < \omega$. Let us define a formula $\sigma_{n + 1}$ claiming that $\exists$ always has a proper response for a network $\mathcal{N}$ being played.

    $\sigma_{n + 1}(\mathcal{N}, v)$ is defined as follows:
    \begin{center}
      $\sigma_{n + 1}(\mathcal{N}, v) = {\sigma_{n + 1}}_c(\mathcal{N}, v) \land {\sigma_{n + 1}}_w(\mathcal{N}, v) \land {\sigma_{n + 1}}_j(\mathcal{N}, v)$
    \end{center}
    where
    \begin{itemize}
      \item ${\sigma_{n + 1}}_c(\mathcal{N}, v) = \bigwedge \limits_{ \substack{x, y, z \in \operatorname{Nodes}(\mathcal{N}) \\ t_1 \in l_{\mathcal{N}}(x, y) \\ t_2 \in l_{\mathcal{N}}(y, z)}} \sigma_{n}(\mathcal{N}_c(x, z, t_1, t_2), v)$

      \item ${\sigma_{n + 1}}_w(\mathcal{N}, v) = \bigwedge \limits_{\substack{x,y \in \operatorname{Nodes}(\mathcal{N}) \\ t \in l_{\mathcal{N}}(x,y)}} \forall u_1, u_2 (t \leq u_1 ; u_2 \rightarrow \bigvee \limits_{w \in \operatorname{Nodes}(\mathcal{N}) \cup \{ z \}} \mathcal{N}_c(x,y,w, u_1, u_2))$, where $z \notin \operatorname{Nodes}(\mathcal{N})$.

      \item ${\sigma_{n + 1}}_j(\mathcal{N}, v) = \forall a \: \forall b (v = a + b \rightarrow \bigwedge \limits_{x, y \in \operatorname{Nodes}(\mathcal{N})} \sigma_{n}(\mathcal{N}_{j_1}(\mathcal{N}, x, y, a), v) \lor \sigma_{n}(\mathcal{N}_{j_2}(\mathcal{N}, x, y, b), v))$
    \end{itemize}
  \end{enumerate}

  So, $\exists$ has a winning strategy iff these formulas are true under the valuation $\vartheta$ since the formulas $\{ \sigma_n \}_{n < \omega}$ encode the presence of a winning strategy for $\exists$ in every finite round.

  Let $v_0$ be any variable, $\mathcal{N}_{v_0}$ denotes the term network having the form
  $\langle \{ \{ x_0, x_1 \}, \{ (x_0, x_1) \}, l \} \rangle$, where $l(x,y) = \{ v_0 \}$.
  We define the following sequence of formulas $(\rho_n)_{n < \omega}$:
  \begin{center}
    $\rho_n = \forall v_0 \forall v_1 (\neg (v_0 \leq v_1) \to \sigma(\mathcal{N}_{v_0}, v_0))$
  \end{center}
\end{proof}

This inductive sequence of formulas provides us the explicit axiomatisation of the class of representable join semilattice-ordered semigroups.

\begin{theorem}\label{axiomatisation}
  A join semilattice-ordered semigroup $\mathcal{A}$ is representable iff $\mathcal{A} \models \{ \rho_n \}_{n < \omega}$. Moreover, ${\bf R}(;,+)$ has a recursively enumerable universal axiomatisation.
\end{theorem}

\begin{proof}
  Let us define a two sorted language with sorts ${\bf a}$ (algebra) and ${\bf r}$ (representation). ${\bf R}(;,+)$ clearly forms a pseudo-elementary class, see \cite[Introduction]{hirsch2007representable} for more details. Moreover, this class is pseudo-universal and it satisfies the condition of the second item of Theorem~\ref{axiomatise}.

  By Proposition~\ref{fin}, Lemma~\ref{repr}, and Lemma~\ref{ax}, a countable join semilattice-ordered semigroup $\mathcal{A}$ is representable iff $A \models \{ \rho_n \}_{n < \omega}$. Suppose $\mathcal{A}$ is uncountable. The class is pseudo-elementary, so it is closed under elementary equivalence, so, by the downward L\"{o}wenheim-Skolem theorem \cite[Corollary 3.1.5]{hodges1993model}, we can take $\mathcal{A}_0 \preceq \mathcal{A}$, a countable elementary substructure of $\mathcal{A}$. Then $\mathcal{A}_0 \models \{ \rho_n \}_{n < \omega}$
  iff $\mathcal{A} \models \{ \rho_n \}_{n < \omega}$. Therefore, if $\mathcal{A}_0$ is representable, so is $\mathcal{A}$.
\end{proof}

As we have already discussed, the finite representation property for $(;,+)$-structures remains an open question. If the solution is positive, then the problem of representability for finite join semilattice-ordered semigroups is decidable since finite representability and recursive axiomatisability imply decidability.

\section{Acknowledgements}

The author would like to thank Robin Hirsch,  Ian Hodkinson, Stepan Kuznetsov, Ja\v{s} \v{S}emrl, Valentin Shehtman, and his supervisor Ilya Shapirovsky for valuable comments. The author is also grateful to the reviewers whose comments improved the original version of the paper.

\bibliographystyle{splncs04}
\bibliography{Paper}
\end{document}